\documentclass[11pt,leqno]{amsart}

\usepackage{graphicx}
\usepackage{amsfonts,delarray,amssymb,amsmath,amsthm,a4,a4wide}
\usepackage{latexsym}
\usepackage{epsfig}
\usepackage{color}
\usepackage[margin=1in]{geometry}

\newcommand{\trace}{\text{trace}}

\newcommand{\R}{{\mathbb R}} 

\newcommand{\e}{\varepsilon}

\newcommand{\p}{\partial}

\newenvironment{myindentpar}[1]%
{\begin{list}{}%
         {\setlength{\leftmargin}{#1}}%
         \item[]%
}
{\end{list}}

\theoremstyle{plain}
\newtheorem{thm}{Theorem}[section]

\newtheorem{lem}[thm]{Lemma}
\newtheorem{conj}[thm]{Conjecture}

\newtheorem{prop}[thm]{Proposition}
\newtheorem{rem}[thm]{Remark}

\theoremstyle{definition}
\newtheorem{defn}[thm]{Definition}

\title[ Polynomial decay in $W^{2,\varepsilon}$ estimates ]{Polynomial decay in $W^{2,\varepsilon}$ estimates for viscosity supersolutions of fully nonlinear elliptic equations}
\author{Nam Q. Le}
\address{Department of Mathematics, Indiana University,
Bloomington, 831 E 3rd St, IN 47405, USA.}
\email{nqle@indiana.edu}
\thanks{The research of the author was supported in part by the National Science Foundation under grant DMS-1764248.}

\makeatletter
\def\l@subsection{\@tocline{2}{0pt}{2.5pc}{5pc}{}}
\makeatother

\begin{document}

\begin{abstract}
We prove $W^{2,\varepsilon}$ estimates for viscosity supersolutions of fully nonlinear, uniformly elliptic equations where $\varepsilon$ decays polynomially with respect to the ellipticity ratio of the equations.
Our result is related to a conjecture of Armstrong-Silvestre-Smart [{Comm. Pure Appl. Math.} {\bf 65} (2012), no. 8, 1169--1184] which predicts a linear decay for $\varepsilon$ with respect to the ellipticity ratio of the equations.
\end{abstract}
 \maketitle

\section{Introduction and statement of the main result}

In this paper, we prove $W^{2,\e}$ estimates for viscosity supersolutions of fully nonlinear, uniformly elliptic equations where $\e$ decays polynomially with respect to the ellipticity ratio of the equations.

Let us recall some history and motivation for these estimates. $W^{2,\e}$ estimates for strong solutions of linear, uniformly elliptic equations in nondivergence form with only 
measurable coefficients were first obtained by Lin~\cite{Lin}. The positive exponent $\e$ is small and depends only on the dimension and the ellipticity of the equations.
Around the same time, Evans \cite{E} discovered similar $W^{2,\e}$ estimates for fully nonlinear, uniformly elliptic equations of the form $F(D^2 u) = 0.$
$W^{2,\e}$ estimates were later extended to viscosity solutions of fully nonlinear, uniformly elliptic equations in Caffarelli and Cabr\'e \cite{CC}. In \cite{GT}, Guti\'errez and Tournier obtained
$W^{2,\e}$ estimates for the linearized Monge-Amp\`ere equation which is in general degenerate and singular. Recently, 
Lin's approach has been extended by Yu \cite{Yu} to establish $W^{\sigma,\e}$ estimates for a class of nonlocal fully nonlinear elliptic equations.

Combining the $W^{2,\e}$ estimates in \cite{CC, Lin} with a deep result of Savin \cite{Sa} on the $C^{2,\alpha}$ regularity of viscosity solutions of fully nonlinear, uniformly elliptic equations which are close to quadratic polynomials, Armstrong, Silvestre and Smart \cite{ASS} proved a partial regularity result for viscosity solutions of general fully nonlinear, uniformly elliptic equations together with an estimate on the Hausdorff dimension of the singular set. The important point in \cite{ASS} is that no convexity nor concavity  is assumed of the equations. More precisely, they proved that a viscosity solution of a uniformly elliptic, fully nonlinear equation $F(D^2 u)=0$ in a domain $\Omega\subset\R^n$ is $C^{2,\alpha}$ on the compliment of a closed set $\Sigma\subset\overline{\Omega}$ of Hausdorff dimension at most $n-\e$. The function $F$ is assumed to be $C^1$ and uniformly elliptic with ellipticity constants $\lambda$ and $\Lambda$, and the constant $\e> 0$ is exactly the exponent in $W^{2,\e}$ estimates for viscosity supersolutions of  fully nonlinear elliptic equations with ellipticity constants $\lambda$ and $\Lambda$; see Proposition \ref{w2ep} for a precise statement. As remarked in \cite[Remark 5.4]{ASS},
the dimension of the singular set in the partial regularity result in \cite{ASS} could be further reduced if we could improve the exponent $\e$ of the $W^{2, \e}$ estimates. In this paper, we offer one such improvement
 from the known lower bound for $\e$ which decays exponentially with respect to the ellipticity ratio of the equations to a new lower bound which decays polynomially. For further discussion, we introduce some standard notation.

Throughout, let $n\geq 2$ be a positive integer. Let $\mathcal{S}_n$ denote the set of real $n\times n$ symmetric matrices. Let $I_n= (\delta_{ij})_{1\leq i, j\leq n}\in\mathcal{S}_n$ be the identity matrix.
Recall that the Pucci extremal operators (see, for example, \cite[Chapter 2]{CC} and \cite[Chapter 5]{HL}) are defined for constants $0 < \lambda \leq \Lambda$ and $
M\in \mathcal{S}_n$ by
\begin{equation*}
\mathcal{M}^{+}_{\lambda,\Lambda}(M) : =  \sup_{\lambda I_n \leq A \leq \Lambda I_n} \trace (AM) ~ \mbox{and} \quad \mathcal{M}^{-}_{\lambda,\Lambda}(M) : =  \inf_{\lambda I_n \leq A \leq \Lambda I_n} \trace (AM).
\end{equation*}
We denote by $Q_{r}(x):=\{ y \in \R^n: |y_i - x_i| < \frac{r}{2} \}$ the open cube centered at $x$ and of side length $r$. 
Denote by $B_r(x): = \{ y\in \R^n: |y-x| < r\}$ the ball of radius $r$ centered at $x$.
For simplicity, we set $Q_{r} := Q_{r}(0)$
and $B_r: = B_r(0)$. We denote the n-dimensional Lebesgue measure of a measurable set $E\subset\R^n$ by $|E|$.\\
\noindent
Given a domain $\Omega \subseteq \R^n$ and a function $u\in C(\Omega)$, define the quantities
\begin{multline*}
 \underline \Theta(u,\Omega)(x): = \inf\big{\{} A \geq 0 : \ \mbox{there exists} \ p \in \R^n \ \mbox{such that for all} \ y \in \Omega, \\ 
u(y) \geq u(x) + p\cdot (y-x) - \tfrac12 A|x-y|^2 \big{\}}.
\end{multline*}
The quantity $\underline \Theta(u,\Omega)(x)$ is the minimum curvature of any paraboloid that touches $u$ from below at $x$. If $u$ cannot be touched from below at $x$ by any paraboloid, then $\underline \Theta(u,\Omega)(x) = +\infty$. 

\medskip

Armstrong-Silvestre-Smart proved the following
$W^{2,\e}$ estimates for viscosity supersolutions; see \cite[Proposition 3.1]{ASS}. 
\begin{prop}[\cite{ASS}]
 \label{w2ep} Let $\lambda\leq\Lambda$ be positive constants.
If $u \in C(B_1)$ satisfies the inequality
$
\mathcal{M}^{-}_{\lambda,\Lambda}(D^2u) \leq 0 $ in $ B_1\subset\R^n,$
then 
\begin{equation} \label{w2epest}
\left| \left\{ x\in B_{1/2} : \underline\Theta(u,B_1)(x) > t\right\} \right| \leq Ct^{-\e}
\end{equation}
for all $t>t_0 \sup_{B_1} |u|$,
where the constants $C,t_0, \e > 0$ depend only on $n,\lambda$ and $\Lambda$.
\end{prop}
We refer the reader to \cite{CC} for more on viscosity solutions. A similar result to Proposition \ref{w2epest} was obtained in \cite[Lemma 7.8]{CC} and \cite[Lemma 5.15]{HL}.
Obviously \eqref{w2epest} implies that for any $0< \hat \e < \e$,
\begin{equation*}
\int_{B_{1/2}} \left( \underline\Theta(u,B_1)(x)\right)^{\hat\e}\, dx \leq C \sup_{B_1} |u|^{\hat\e},
\end{equation*}
where the constant $C$ depends additionally on a lower bound for $\e - \hat \e$. As emphasized by authors in \cite{ASS}, the precise form of the estimate in Proposition \ref{w2epest} 
which involves the quantity  $\underline\Theta(u,B_1)(x)$ is crucial in their proof of the partial regularity result
for viscosity solutions of general fully nonlinear, uniformly elliptic equations. In fact, the weaker statement that $u$ is merely twice differentiable at almost every point with $|D^2u|\in L^{\e}$ is insufficient to prove their partial regularity result. 

By constructing an explicit example \cite[Remark 3.3]{ASS}, Armstrong-Silvestre-Smart showed that the exponent $\e$ in Proposition \ref{w2epest} cannot be larger than $2(\Lambda/\lambda+1)^{-1}$. They made the following conjecture:
\begin{conj} \label{ASSconj}
The optimal exponent in Proposition~\ref{w2ep} is $\e = 2(\Lambda/\lambda+1)^{-1}$.
\end{conj}
Except for the case $\lambda=\Lambda$ for which Conjecture \ref{ASSconj} is known to be true, it is widely open for the case $\lambda<\Lambda$. In this paper, we will focus on this case, especially when
$\Lambda/\lambda$ is large.

Known estimates for $\e$ (see the discussion at the end of this Introduction) give that $\e$ decays exponentially with respect to $\Lambda/\lambda$. Although we are unable to prove Conjecture \ref{ASSconj}, we prove that $\e$ decays at most polynomially with respect to $\Lambda/\lambda$. 
Roughly speaking, our estimates imply that
$$\e>(\Lambda/\lambda)^{-(n+1)}c(n)$$
for some positive constant $c(n)$ depending only on the dimension $n$.
Before stating our theorem, we introduce the following definition.
\begin{defn}\label{GMs}
For  $v\in C(\Omega)$ and $K>0$, we define the sets
\begin{equation*}
G_K^-(v,\Omega)
 = \big\{\bar x\in \Omega:  \text{there is } p\in\R^n \text{ such that } v(x)\geq v(\bar x) +
p\cdot (x-\bar x) -
\frac{K}{2}|x-\bar x|^2\, \forall x\in \Omega\big\}
\end{equation*}
and
\begin{equation*}A^{-}_K(v,\Omega)=\Omega\setminus G^{-}_K(v,\Omega).
\end{equation*}
\end{defn}
We observe from the definitions that
$$ \left\{ x\in B_{1/2} : \underline\Theta(u,B_1)(x) > t\right\} \subset A_t^{-}(u,B_1)\cap B_{1/2}.$$
Our main theorem states:
\begin{thm}
\label{improve2} Let $\lambda\leq\Lambda$ be positive constants.
 If $u \in C(B_1)$ satisfies the inequality
$
\mathcal{M}^{-}_{\lambda,\Lambda}(D^2u) \leq 0 $ in $ B_1\subset\R^n,$
then 
\begin{equation} 
\label{w2delta}
\left|  A_t^{-}(u,B_1)\cap B_{1/2} \right| \leq Ct^{-\e}
\end{equation}
for all $t>t_0 \sup_{B_1} |u|$,
where the constants $C,t_0, \e$ depend only on $n,\lambda$ and $\Lambda$ with 
\begin{equation}
\label{e_est}
\e>\left(\frac{\lambda}{\lambda + (n-1)\Lambda}\right)^n\left(
\frac{1}{4\sqrt{n}}\right)^n\frac{1}{\log  \left[10^5 n^3 (36n)^{\max\{1, \frac{(n-1)\Lambda}{\lambda}-1\}}\right]}.
\end{equation}
Thus, the exponent $\e$ in Proposition~\ref{w2ep} satisfies
$\e>(\Lambda/\lambda)^{-(n+1)}c(n).
$
\end{thm}
From Theorem \ref{improve2} together with \cite[Lemma 5.2]{ASS} and \cite[Proposition 1.2]{CS}, we can conclude that the exponent $\e$ in the $W^{3,\e}$ estimates for viscosity solutions of general fully nonlinear, uniformly elliptic equations
decays polynomially with respect to the ellipticity ratio of the equations.

For strong supersolutions of linear, uniformly elliptic equations in nondivergence form, we lower the power $(n+1)$ in Theorem \ref{improve2} to $n$ as in the following theorem.
\begin{thm}
\label{improve3}
Let $\lambda\leq\Lambda$ be positive constants.
Assume $(a^{ij}(x))\in \mathcal{S}_n$ satisfies 
$\lambda I_n\leq (a^{ij}(x))\leq \Lambda I_n$ a.e. in $B_1\subset\R^n$.
If $u \in W^{2,n}(B_1)$ satisfies the inequality
$ a^{ij} u_{ij} \leq 0 $ in $ B_1,$
then 
\begin{equation*} 
\left|  A_t^{-}(u,B_1)\cap B_{1/2} \right| \leq Ct^{-\e}
\end{equation*}
for all $t>t_0 \sup_{B_1} |u|$,
where the constants $C,t_0, \e$ depend only on $n,\lambda$ and $\Lambda$ with $$\e>\left(\frac{\lambda}{\Lambda}\right)^{n-1}\left(
\frac{1}{4\sqrt{n}}\right)^n\frac{1}{\log  \left[10^5 n^3 (36n)^{\max\{1, \frac{(n-1)\Lambda}{\lambda}-1\}}\right]} >(\Lambda/\lambda)^{-n}c(n).$$
\end{thm}
We indicate how to prove the $W^{2,\e}$ estimates together with numerical improvement on $\e$. The heart of $W^{2,\e}$ estimates is the following measure and localization estimate.
\begin{lem}
\label{touchCC}
Assume that $\overline{B}_{2\sqrt{n}}\subset\Omega\subset\R^n.$
Suppose that $v\in C(\Omega)$ satisfies
 $\mathcal{M}^{-}_{\lambda,\Lambda} (D^2 v)\leq 0$ in $\Omega$.
If $G_1^{-}(v,\Omega)\cap Q_3\neq \emptyset$ then there is $M(n,\lambda,\Lambda)>1$ and $\sigma(n,\lambda,\Lambda)\in (0,1)$ such that
$$|G_{M}^{-}(v,\Omega)\cap Q_1|\geq 1- \sigma.$$
\end{lem}
Given Lemma \ref{touchCC}, we find that the exponent $\e$ in the $W^{2,\e}$ estimates can be taken to be $\frac{\log \frac{1}{\sigma}}{\log M}$; see Theorem \ref{W2delta_thm}. When $\sigma $ is small, $\e\approx \frac{1-\sigma}{\log M}$.
A careful tracing of the constants in the proofs of $W^{2,\e}$ estimates in \cite[Proposition 7.4]{CC}, \cite[Lemma 5.15]{HL} and \cite{ASS} reveals that, for a fixed dimension $n$, the exponent $\e$ decays exponentially 
with respect to the ratio $\frac{\Lambda}{\lambda}$ of the ellipticity constants of the equations. One of the reasons comes from the use of the Aleksandrov-Bakelman-Pucci (ABP) maximum principle \cite[Theorem 3.2]{CC} applied to the barrier constructed in \cite[Lemma 4.1]{CC}.
This application gives the measure estimate in \cite[Lemma 4.5 and Lemm 7.5]{CC} and also in Lemma \ref{touchCC} together with the value of $\sigma$ in Lemma \ref{touchCC} of the form 
$\sigma\approx 1-\frac{1}{M}$ for $M\approx e^{(n-1)\frac{\Lambda}{\lambda}}$ when $\frac{\Lambda}{\lambda}$ is large.

Our polynomial decay for $\e$ in Theorems \ref{improve2} and \ref{improve3} comes from an improvement of $(1-\sigma)$ in the measure estimate; see Lemma \ref{meas_lem}. To obtain the measure estimate, we use the method of sliding paraboloids and the area formula as in \cite{Cab, L, Sa} to bypass the ABP estimate. An important feature of our measure estimate is that $\sigma$ can be estimated {\it independently of $M$}. Moreover, it can potentially be applicable to singular and degenerate elliptic equations as in the case of the Harnack inequality in \cite{L}. The constant $M$
comes from the localization Lemma \ref{construct_lem}. Its proof, which is based on the construction of a suitable subsolution, is standard; see also \cite[Lemma 4.1]{CC} and \cite[Lemma 5.13]{HL}.

We have tried to make explicit all constants in our estimates. Obviously, there are lot of rooms for improvement of their numerical values. It would be interesting to lower the exponent $(n+1)$ in the decay rate $(\Lambda/\lambda)^{-(n+1)}$ in Theorem \ref{improve2}
and $n$ in Theorem \ref{improve3}.

The rest of the paper is organized as follows. In Section \ref{measure_sec}, we prove a measure estimate in Lemma \ref{meas_lem} for viscosity supersolutions
 and in Lemma \ref{st_meas_lem} for strong supersolutions
 and a localization result in Lemma \ref{construct_lem}.
The proofs of Theorems \ref{improve2} and \ref{improve3} will be given in Section \ref{pf_sec}.
\section{Measure estimate and localization}
\label{measure_sec}

Throughout this section, $\lambda\leq\Lambda$ are positive constants.

Our first lemma is a measure estimate. It roughly says that if a viscosity supersolution can be touched from below at a point in a small cube by a paraboloid of some fixed opening then it can be touched
from below at a set of positive measure in a larger cube by paraboloids of larger opening. More precisely, it states as follows.
\begin{lem}[Measure estimate for viscosity supersolutions]\label{meas_lem}  
Assume that $\overline{Q}_2\subset \Omega\subset\R^n$.
Suppose that $v\in C(\Omega)$ satisfies
 $\mathcal{M}^{-}_{\lambda,\Lambda} (D^2 v)\leq 0$ in $\Omega$.
Assume that 
 $G_{1/n}^{-}(v,\Omega)\cap Q_{\frac{1}{4\sqrt{n}}}\neq \emptyset.$
Then 
 $$|G^{-}_{32}(v,\Omega)\cap Q_1|\geq (1-\sigma) |Q_1|$$
 for
 $$\sigma:=
1- \left(\frac{\lambda}{\lambda + (n-1)\Lambda}\right)^n\left(\frac{1}{4\sqrt{n}}\right)^n.$$
 \end{lem}
\begin{proof}[Proof of Lemma \ref{meas_lem}] 
For simplicity, we denote
$$\alpha_1=\frac{1}{4\sqrt{n}}<\frac{1}{4}.$$
{\it Step 1:} We first consider the case when $v$ is uniformly semiconcave in $Q_{3/2}$, that is, the graph of $v$ admits at all points in $Q_{3/2}$ a touching paraboloid of opening $m$ from above.

From $G_{1/n}^{-}(v,\Omega)\cap Q_{\alpha_1}\neq \emptyset$, we can find an affine function $L(x)$ such that
$$v(x) \geq L(x)-\frac{1}{2n}|x-x^\ast|^2~\text{for all } x\in\Omega~\text{with equality at } x^\ast\in Q_{\alpha_1}.$$
By considering $v-L +1$ instead of $v$, we can assume that
$$v(x) \geq 1-\frac{1}{2n}|x-x^\ast|^2~\text{for all } x\in\Omega~\text{with equality at } x^\ast\in Q_{\alpha_1}.$$
Consider the set of vertices $V=Q_{\alpha_1}$. As in \cite{Sa},
 for each $y\in V$, we slide 
 the  paraboloids $$ -\frac{K}{2} |x-y|^2 + C_y$$ of opening $K>0$ until they touch
 the graph of $v$ from below at some point $x\in \overline Q_1$, called the contact point.
 We define the contact set by
\begin{equation*}E_K (V, Q_1, v)=\{ x\in \overline{Q_1}: \text{there is } y\in V~\text{such that } \inf_{\overline{Q_1}} \left(v + \frac{K}{2} |\cdot-y|^2\right)= v(x)
+ \frac{K}{2} |x-y|^2\}.
\end{equation*}
{\bf Claim 1.}  With $K=32$, we have the following:
\begin{equation}
\label{insideQ}
E_K (V, Q_1, v)\subset Q_1,
\end{equation}
\begin{equation}
\label{D2_contact}
E_K (V,  Q_1, v) \subset G^{-}_{K}(v, \Omega).
\end{equation}

Indeed, for each $y\in V$, we consider the function
 $$P(x) = v(x) + \frac{K}{2} |x-y|^2$$
 and look for its minimum points on $\overline{Q_1}$. 

If $x\in\p Q_1$, then $|x-y|\geq \frac{1-\alpha_1}{2}>\frac{3}{8}$ and hence 
 \begin{equation}P(x)\geq 1-\frac{1}{2n}|x-x^\ast|^2  + \frac{K}{2} |x-y|^2 > \frac{K}{2} |x-y|^2   > 16 \frac{3^2}{8^2}>2.
  \label{PbdrS1}
 \end{equation}
Note that 
$|x^\ast-y|^2 < n\alpha_1^2 =\frac{1}{16}.$
Therefore
 \begin{equation}P(x^\ast)= v(x^\ast) + \frac{K}{2} |x^\ast-y|^2= 1 + 16 |x^\ast-y|^2<2 .
\label{Px0}  
 \end{equation}
From (\ref{PbdrS1}) and (\ref{Px0}), we deduce that $P$ attains its minimum on $\overline{Q_1}$ at a point $x\in Q_1$. 
Hence  $E_K (V, Q_1, v)\subset Q_1$, proving (\ref{insideQ}).

It remains to prove (\ref{D2_contact}). For each contact point $x\in E_K (V, Q_1, v)\subset Q_1$, let $y\in V$  be such that
\begin{equation} 
\label{contact_eq}
v(x)
+ \frac{K}{2} |x-y|^2 \leq v(z)
+ \frac{K}{2} |z-y|^2~\text{for all~} z\in\overline{Q_1}.
\end{equation}
We show that
$x\in G^{-}_{K}(v, \Omega)$
and consequently, (\ref{D2_contact}) holds. For this, it is crucial to note that (\ref{contact_eq}) also holds for all $z\in \Omega$, that is,
\begin{equation}
\label{contact_eq2}
v(x)
+ \frac{K}{2} |x-y|^2\leq v(z)
+ \frac{K}{2} |z-y|^2~\text{for all~} z\in\Omega.
\end{equation}
Indeed, it suffices to verify (\ref{contact_eq2}) for $z\in \Omega\backslash Q_1$. In this case, we use
\begin{equation}
 1-\frac{1}{2n}|z-x^\ast|^2  + \frac{K}{2} |z-y|^2 >2.
\label{convex_arg}
\end{equation}
Indeed, let
$$W=\{z\in \Omega: 1-\frac{1}{2n}|z-x^\ast|^2  + \frac{K}{2} |z-y|^2\leq 2\}.$$
It suffices to show that $W\subset Q_1.$ Indeed, we first note that 
$W$ is convex and $y\in W.$ If $z\in \p Q_1$ then  by (\ref{PbdrS1}), we have $z\not\in W$. Thus, the convexity of $W$ implies that $W\subset Q_1$.

Now, consider $z\in \Omega\backslash Q_1$. Then,  in view of (\ref{Px0}) and (\ref{convex_arg}), we find that (\ref{contact_eq2}) follows from
$$v(z) +  \frac{K}{2} |z-y|^2\geq 1-\frac{1}{2n}|z-x^\ast|^2  + \frac{K}{2} |z-y|^2
> 2> P(x^\ast) \geq v(x)
+ \frac{K}{2} |x-y|^2.$$
By the minimality of $P$ at $x$ (see (\ref{contact_eq2})), we have 
$D v(x) + K(x-y)=0$
which gives
\begin{equation}
 \label{uuv}y = x +\frac{1}{K}D v(x)
\end{equation}
From the minimality of $P$ at $x$, we also have 
\begin{equation}
\label{uvD2}
D^2 v(x) \geq -K I_n.
\end{equation}
From (\ref{contact_eq2}) and (\ref{uuv}), we deduce that for all $z\in\Omega$, 
\begin{eqnarray*}
v(z)\geq v(x) -\frac{K}{2} (|z|^2-|x|^2) + K y\cdot (z-x) 
= v(x) + Dv(x) \cdot(z-x) - \frac{K}{2}|z-x|^2.
\end{eqnarray*}
Therefore $x\in G^{-}_{K}(v,\Omega)$, completing the proof of Claim 1.

Before proceeding further, we note from the proof of (\ref{insideQ}) that for each $y\in V$, there is $x\in E:=E_K (V, Q_1, v)$ such that (\ref{uuv}) holds, that is $y= \Phi(x)$ where $$\Phi(x) = x + \frac{1}{K}D v(x).$$
 It follows that $V\subset \Phi(E).$ It is easy to see that $\Phi$ is Lipschitz on $E$ with Lipschitz constant bounded by C(m). By (\ref{uvD2}), we have
 $$D\Phi(x) =I_n + \frac{1}{K}D^2 v(x)\geq 0~\text{on } E.$$
 Moreover,  by definition, $E$ is a closed set and thus measurable.
 By the area formula, we have
 \begin{equation}
 \label{CoA}
 |V|\leq |\Phi(E)|= \int_{E} \det D\Phi (x) dx = \int_{E} \det (I_n +\frac{1}{K} D^2 v (x)) dx.
 \end{equation}
It remains to estimate from above the integrand in (\ref{CoA}).

Denote by $\mathcal{N}$ the set of points $x\in\Omega$ for which $v$ can be approximated by a quadratic polynomial near $x$, that is,
\begin{equation}
\label{v_Taylor}
v(z) = P(x, z) + o(|z-x|^2),
\end{equation}
where $$P(x, z)= v(x) + p(x)\cdot (z-x) + \frac{1}{2} (z-x)^{T} M(x) (z-x); M(x)\in \mathcal{S}_n.$$
Since $v$ is semi-concave, the Aleksandrov theorem (see \cite[Section 6.4]{EG}) tells us that $$|\Omega\backslash \mathcal{N}|=0.$$
{\bf Claim 2.} If $x\in E_K (V, Q_1, v)\cap \mathcal{N}$ then 
\begin{equation}
\label{D2v_ineq}
-K I_n \leq D^2 v(x) =M(x)\leq K\frac{(n-1)\Lambda}{\lambda} I_n.
\end{equation} 
The left inequality of (\ref{D2v_ineq}) follows from (\ref{uvD2}). It remains to prove the inequality on the right hand side of (\ref{D2v_ineq}).
From (\ref{v_Taylor}), we know that for all $\delta>0$ small, $$P(x, z)-\frac{\delta}{2}|z-x|^2 + const$$ touches $v(z)$ from below in a neighborhood of $x$ at some point $\tilde x$. Since $v$ is a viscosity supersolution, we find
\begin{equation}
\label{D2v_eq}
\mathcal{M}^{-}_{\lambda,\Lambda} (M(x)-\delta I_n)\leq 0.
\end{equation}
Assume by contradiction that the largest eigenvalue of $M(x)$ is $C>K\frac{(n-1)\Lambda}{\lambda}.$
Then, from (\ref{D2v_eq}) and the definition of $\mathcal{M}^{-}_{\lambda,\Lambda}$, we find that
$ \lambda (C-\delta)-(n-1) \Lambda (K +\delta)\leq 0.$
By letting $\delta\rightarrow 0$, we obtain
$C \leq K\frac{(n-1)\Lambda}{\lambda}$ and hence a contradiction with $C>K\frac{(n-1)\Lambda}{\lambda}.$
Thus, (\ref{D2v_ineq}) is proved.

From (\ref{D2v_ineq}), we find that for $x\in E_K (V, Q_1, v)\cap \mathcal{N}$,
\begin{equation}
\label{west}
\det (I_n + \frac{1}{K} D^2 v(x))  \leq  \left(1+ \frac{(n-1)\Lambda}{\lambda}\right)^n.
\end{equation}
Using (\ref{CoA}) and (\ref{west}), we get
\begin{eqnarray*} |V|\leq \int_E  \det (I_n + \frac{1}{K} D^2 v(x)) dx= \int_{E\cap \mathcal{N}}  \det (I_n + \frac{1}{K} D^2 v(x))dx
\leq  \left(1+ \frac{(n-1)\Lambda}{\lambda}\right)^n |E\cap \mathcal{N}|.
 \end{eqnarray*}
Recalling $V=Q_{\alpha_1}$, it follows that
$$|E|\geq \left(\frac{\lambda}{\lambda + (n-1)\Lambda}\right)^n|V|=\left(\frac{\lambda}{\lambda + (n-1)\Lambda}\right)^n\left(\frac{1}{4\sqrt{n}}\right)^n|Q_1|=: c_0 |Q_1|.$$
Using (\ref{D2_contact}) and $K=32$, the conclusion of the lemma follows with $\sigma=1-c_0$.\\
{\it Step 2:} Now we treat the general case without assuming that $v$ is semiconcave. For this, we regularize $v$ by the standard method of inf-convolution. Let
$$v_\delta(x) = \inf_{y\in\overline{Q}_2}\left\{ u(y) +\frac{1}{\delta}|y-x|^2\right\}, x\in Q_2.$$ 
It is easy to check that $v_\delta$ is semiconcave and $v_\delta\rightarrow v$ uniformly on compact subsets of $Q_2$. Moreover $\mathcal{M}^{-}_{\lambda,\Lambda} (D^2 v_\delta)\leq 0$ in $Q_{3/2}$; see, for example, the remark after Theorem 5.1 in \cite{CC}.
By the above proof, we find
$$|E_\delta| \geq c_0|Q_1|$$
where $E_\delta$ is the corresponding touching set for $v_\delta$, that is, $E_\delta= E_K(V, Q_1,v_\delta)$. It is easy to check that
$$\displaystyle \limsup E_{1/k}=\bigcap_{m=1}^{\infty}\bigcup_{k=m}^{\infty} E_{1/k}\subset E.$$
Thus we conclude that $|E|\geq c_0|Q_1|.$
\end{proof}

For strong supersolutions, we have the following measure estimate.
\begin{lem}[Measure estimate for strong supersolutions]\label{st_meas_lem}  
Assume that $\overline{Q}_2\subset \Omega\subset\R^n$.
Assume $(a^{ij}(x))\in \mathcal{S}_n$ satisfies 
$\lambda I_n\leq (a^{ij}(x))\leq \Lambda I_n$ a.e. in $\Omega$.
Suppose that $v \in W^{2,n}(\Omega)$ satisfies the inequality
$ a^{ij} v_{ij} \leq 0 $ in $ \Omega$.
Assume that 
 $G_{1/n}^{-}(v,\Omega)\cap Q_{\frac{1}{4\sqrt{n}}}\neq \emptyset$.
Then 
 $$|G^{-}_{32}(v,\Omega)\cap Q_1|\geq (1-\sigma_1) |Q_1|
\text{ for }
 \sigma_1:=
1- \left(\frac{\lambda}{\Lambda}\right)^{n-1}\left(\frac{1}{4\sqrt{n}}\right)^n.$$
 \end{lem}
\begin{proof}[Proof of Lemma \ref{st_meas_lem}]
The proof is similar to that of Lemma \ref{meas_lem}. Instead of (\ref{west}), we have the improved estimate:
\begin{equation}
\label{st_west}
\det (I_n + \frac{1}{K} D^2 v(x))  \leq  \left(\frac{\Lambda}{\lambda}\right)^{n-1}~\text{for all } x\in E_K (V, Q_1, v)\cap \mathcal{N}.
\end{equation}
We indicate how to obtain this estimate. For $x\in E_K (V, Q_1, v)\cap \mathcal{N}$, we have $I_n + \frac{1}{K} D^2 v(x)\geq 0$
and from $a^{ij}(x) v_{ij}(x)\leq 0$, we find that
$a^{ij}(x) (\delta_{ij} + \frac{1}{K} v_{ij}(x) )\leq \trace (a^{ij}(x)).$
Using the inequality
\begin{equation}\trace (AB)\geq n(\det A)^{1/n} (\det B)^{1/n}~\text{for } A, B\geq 0~\text{in } \mathcal{S}_n,
\end{equation}
we obtain
\begin{equation}
\label{detA}
 \trace (a^{ij}(x)) \geq n (\det (a^{ij}(x)))^{1/n} (\det (I_n + \frac{1}{K} D^2 v(x)))^{1/n} ~\text{for all } x\in E_K (V, Q_1, v)\cap \mathcal{N}.
\end{equation}
Let $\lambda_1(x)\leq \lambda_2(x)\leq\cdots\leq\lambda_n(x)$ be the eigenvalues of $(a^{ij}(x))$. Then $\lambda_i(x)\in [\lambda,\Lambda]$ for all $i=1,\cdots, n$. We estimate
$$\frac{ \trace (a^{ij}(x))}{ n (\det (a^{ij}(x)))^{1/n}}\leq \frac{n\lambda_n(x)}{n (\lambda_1(x)^{n-1} \lambda_n(x) )^{1/n}}=\left(\frac{\lambda_n(x)}{\lambda_1(x)}\right)^{\frac{n-1}{n}}\leq \left(\frac{\Lambda}{\lambda}\right)^{\frac{n-1}{n}}.$$
Now, (\ref{st_west}) follows from (\ref{detA}) and the above estimates.
\end{proof}
The following lemma says that for a bounded, continuous function in a domain $\Omega$ containing $Q_3$, it can be touched from below at a point in $Q_3$ by a paraboloid of opening propositional to its sup norm.
This fact is well known. However, since we would like to keep track all constants in this paper, we write down its precise formulation. 
\begin{lem}
\label{touch1_lem}
Assume that $\overline{Q}_3\subset \Omega\subset\R^n$.
If $v\in C(\Omega)$ with $|v|\leq \frac{1}{4}$ in $ \Omega$ then $G_1^{-}(v,\Omega)\cap Q_3\neq \emptyset$.
\end{lem}

\begin{proof}
Fix $y\in Q_{1/2}$.
Consider the function
 $P(x) = v(x) + \frac{1}{2} |x-y|^2$
 and look for its minimum points on $\overline{Q_3}$.  At $y$, we have 
 $P(y) = v(y) \leq \frac{1}{4}.$ If $x\in\p Q_3$, then $|x-y|>\frac{5}{4}>1$ and hence 
 $P(x)\geq -\frac{1}{4} + \frac{1}{2} |x-y|^2  > \frac{1}{4}.$
It follows that $P$ attains its minimum on $\overline{Q_3}$ at a point $x_0\in Q_3$ with $P(x_0)\leq \frac{1}{4}$. We show that $x_0\in G_1^{-}(v,\Omega)$. To see this, it remains to show that
$P(x_0)\leq P(z)$ for all $z\in\Omega\backslash Q_3$.
Indeed, when $z\in\Omega\backslash Q_3$, we have $|z-y|>\frac{5}{4}>1$ 
and hence
\begin{equation*}
 P(z)=v(z)
+ \frac{1}{2} |z-y|^2 \geq -\frac{1}{4} + \frac{1}{2} |z-y|^2 >\frac{1}{4}\geq P(x_0).
\end{equation*}
\end{proof}

The next lemma is a localization result. 
It roughly says that if a viscosity supersolution can be touched from below at a point in a large cube by a paraboloid of some fixed opening then it can be touched
from below at point in a smaller cube by a paraboloid of larger opening. 
\begin{lem}[Localization for viscosity supersolutions]
\label{construct_lem}
Assume that $\overline{B}_{2\sqrt{n}}\subset\Omega\subset\R^n.$
 Suppose that $v\in C(\Omega)$ satisfies
 $\mathcal{M}^{-}_{\lambda,\Lambda} (D^2 v)\leq 0$ in $\Omega$ and  $G_{\frac{1}{8n}}^{-}(v,\Omega)\cap Q_{3}\neq \emptyset$.
Then the following assertions hold. 
 \begin{myindentpar}{1cm}
 (i) $\inf_{\overline{Q}_{\frac{1}{12\sqrt{n}}}} (v-L+1)\leq M_1$, for an affine function $L$, where $M_1:= 8 (36 n)^{\max\{1,  \frac{(n-1)\Lambda}{\lambda}-1\}}.$\\
 (ii) The set $G_{M_2}^{-}(v,\Omega)\cap Q_{\frac{1}{4\sqrt{n}}}\neq\emptyset$ where
 $ M_2:= 432n M_1.$
\end{myindentpar}
\end{lem}
\begin{proof}
For simplicity, we denote
$$\alpha_1:=\frac{1}{4\sqrt{n}}; \alpha_2:=\frac{1}{12\sqrt{n}};\alpha_3:=\alpha_2/2=\frac{1}{24\sqrt{n}}.$$
Note that $M_2=  \frac{3M_1}{\alpha_2^2}.$
From $G_{\frac{1}{8n}}^{-}(v,\Omega)\cap Q_{3}\neq \emptyset$, we can find an affine function $L(x)$ such that
$$v(x) \geq L(x)-\frac{1}{16n}|x-x^\ast|^2~\text{for all } x\in\Omega~\text{with equality at } x^\ast\in Q_{3}.$$
By considering $v-L +1$ instead of $v$, we can assume that
$$v(x) \geq 1-\frac{1}{16n}|x-x^\ast|^2~\text{for all } x\in\Omega~\text{with equality at } x^\ast\in Q_3\subset  B_{2\sqrt{n}}.$$

We first show that (i) implies (ii). Fix $y\in \overline{Q_{\alpha_2}}$ such that $v(y)\leq M_1$. We consider the function
 $$P(x) = v(x) + \frac{M_2}{2} |x-y|^2$$
 and look for its minimum points on $\overline{Q_{3\alpha_2}}$.  At $y$, we have 
 $P(y) = v(y) \leq M_1.$ If $x\in\p Q_{3\alpha_2}$, then $|x-y|\geq \alpha_2$ and hence 
 $P(x)\geq \frac{M_2}{2} |x-y|^2 \geq \frac{M_2}{2}\alpha_2^2 >M_1.$
It follows that $P$ attains its minimum on $\overline{Q_{3\alpha_2}}$ at a point $x\in Q_{3\alpha_2}$ with $P(x)\leq M_1$. 

Similarly, since $v\geq 0$ in $B_{2\sqrt{n}}$, we easily see that $P(z)>M_1$ for $z\in B_{2\sqrt{n}}.$ We show that $x\in G_{M_2}^{-}(v,\Omega)$.
To conclude the proof of (ii), it remains to show that $P(x)\leq P(z) $ for all $z\in\Omega\backslash  B_{2\sqrt{n}}.$
Indeed, if $z\in \Omega\backslash  B_{2\sqrt{n}} $, then $|z-y|>|z|/2.$
It follows that
\begin{equation*}
P(z)\geq v(z)
+ \frac{M_2}{2} |z-y|^2 \geq 1-\frac{1}{16n}|z|^2 + \frac{M_2}{8}|z|^2\geq  1 + (\frac{M_2}{8}-\frac{1}{16n})4n >M_1\geq P(x).
\end{equation*}

Finally, we prove (i). We argue by contradiction. Suppose that $v> M_1$ in $\overline{Q_{\alpha_2}}.$ Then
$v>M_1$ in $\overline{B_{\alpha_2/2}}$. Note that $Q_3\subset B_{3\sqrt{n}/2}\subset B_{2\sqrt{n}}$.

We will construct a viscosity subsolution $w: B_{2\sqrt{n}}\setminus B_{\alpha_3}\rightarrow \R$ with the following properties:
\begin{myindentpar}{1cm}
(a) $\mathcal{M}^{-}_{\lambda,\Lambda}(D^2 w)\geq 0$ in  $B_{2\sqrt{n}}\setminus B_{\alpha_3}$.\\
(b) $w\leq 0$ on $\p B_{2\sqrt{n}}$, \\
(c) $w\leq M_1$ on $ \p B_{\alpha_3}$.\\
(d) $w\geq 2$ in $B_{3\sqrt{n}/2}\setminus B_{\alpha_3}$.
\end{myindentpar} 
Assuming the existence of $w$, we finish the proof of (i) as follows. 
First, we note that $\mathcal{M}^{-}_{\lambda,\Lambda}(D^2 v-D^2 w)\leq 0$ in $B_{2\sqrt{n}}\setminus B_{\alpha_3}$. To see this, suppose $x_0\in B_{2\sqrt{n}}\setminus B_{\alpha_3}$ and $\varphi \in C^2 (B_{2\sqrt{n}}\setminus B_{\alpha_3})$
be such that $v-w-\varphi$ attains its minimum value at $x_0$. We need to show that  $\mathcal{M}^{-}_{\lambda,\Lambda}( D^2 \varphi(x_0)) \leq 0$. Indeed, by the definition of $v$, we have
$\mathcal{M}^{-}_{\lambda,\Lambda}(D^2 w(x_0) + D^2 \varphi(x_0))\leq 0.$
It follows from (a) that
$$\mathcal{M}^{-}_{\lambda,\Lambda}( D^2 \varphi(x_0)) \leq - \mathcal{M}^{-}_{\lambda,\Lambda}(D^2 w(x_0) )\leq 0.$$
By (b) and (c), we have $v-w\geq 0$ on $\p (B_{2\sqrt{n}}\setminus B_{\alpha_3})$. By the maximum principle for viscosity supersolution, we obtain $v\geq w$ in $B_{2\sqrt{n}}\setminus B_{\alpha_3}$.  Using (d) and the fact that $v>M_1>2$ in  $\overline{B_{\alpha_3}}$, we conclude that $v\geq 2$ in $B_{3\sqrt{n}/2}$. This contradicts the assumption that
$v(x^\ast)=1$ for some $x^\ast\in Q_3\subset B_{3\sqrt{n}/2}.$ Thus we must have $v\leq M_1$ in $\overline{Q}_{\alpha_2}$.

Let us return to constructing $w$ satisfying (a)-(d). Our construction also explains the choice of $M_1$ in the statement of the lemma.
With $u(x):= \frac{1}{2}|x|^2$, we choose $w$ of the form
$$w(x) = C ([u(x)]^{-m}- (2n)^{-m}).$$
where $C$ and $m$ are large positive numbers depending on $n,\lambda,\Lambda$ to be determined. 

Clearly (b) is satisfied. For any $M_1>0$ and $m>0$, the choice of
$C= M_1 (\alpha_3^2/2)^m$
will guarantee that (c) is satisfied. We fix this choice of $C$.
We compute
$$w_{ij}= Cmu^{-m-2}[(m+1) u_i u_j - u u_{ij}]= Cmu^{-m-2} [(m+1) x_i x_j - u\delta_{ij}].$$
The eigenvalues of $D^2 w$ are: $Cmu^{-m-2} (m+\frac{1}{2})|x|^2$ with multiplicity 1 and $-Cmu^{-m-2} \frac{|x|^2}{2}$ with multiplicity $n-1$.
It follows that in $B_{2\sqrt{n}}\setminus B_{\alpha_3}$, we have
\begin{eqnarray*}
\mathcal{M}^{-}_{\lambda,\Lambda}(D^2 w)=  Cmu^{-m-2} |x|^2 [\lambda (m+\frac{1}{2})-\Lambda \frac{n-1}{2}]\geq 0,
\end{eqnarray*}
that is (a) is satisfied, provided that $$m\geq \frac{(n-1)\Lambda}{2\lambda}-\frac{1}{2}.$$

 To obtain (d), we need to choose $M_1$ so that
in  $B_{3\sqrt{n}/2}\setminus B_{\alpha_3}$, we have
\begin{equation}
\label{M1_ineq}
2\leq  M_1 (\alpha_3^2/2)^m  ([u(x)]^{-m}- (2n)^{-m})\equiv w.
\end{equation}
It suffices to choose
$$M_1= \max\{8, \frac{4}{m}\} (\frac{9n}{4\alpha^2_3})^m= \max\{8, \frac{8}{2m}\} (36 n)^{2m}.$$
This is because in 
$B_{3\sqrt{n}/2}\setminus B_{\alpha_3}$, we have
$ ([u(x)]^{-m}- (2n)^{-m}) \geq (\frac{8}{9n})^{m}[1- \frac{9^m}{16^m}]$ and thus, with the above choice of $M_1$, (\ref{M1_ineq}) follows from 
$$2\leq  M_1 (\frac{4\alpha^2_3}{9n})^m \left[1-\frac{9^m}{16^m}\right]= M_1 (\alpha_3^2/2)^m  (\frac{8}{9n})^{m}[1- \frac{9^m}{16^m}]
\leq M_1 (\alpha_3^2/2)^m  ([u(x)]^{-m}- (2n)^{-m}).$$
With $m=\max\{\frac{(n-1)\Lambda}{2\lambda}-\frac{1}{2}, \frac{1}{2}\}$, we have
$M_1= 8 (36 n)^{\max\{1,  \frac{(n-1)\Lambda}{\lambda}-1\}}$, completing the proof of (i).
\end{proof}
Combining Lemmas \ref{construct_lem} (ii) and \ref{meas_lem}, we obtain the following measure and localization result.
\begin{lem}[Measure and localization estimate for viscosity supersolutions]
\label{touch3_lem}
Assume that $\overline{B}_{2\sqrt{n}}\subset\Omega\subset\R^n.$
Suppose that $v\in C(\Omega)$ satisfies
 $\mathcal{M}^{-}_{\lambda,\Lambda} (D^2 v)\leq 0$ in $\Omega$.
If $G_1^{-}(v,\Omega)\cap Q_3\neq \emptyset$ then there is $M=(32n) (8nM_2)$ such that
$$|G_{M}^{-}(v,\Omega)\cap Q_1|\geq (1- \sigma)|Q_1|$$
where
$$\sigma=1-   \left(\frac{\lambda}{\lambda + (n-1)\Lambda}\right)^n\left( \frac{1}{4\sqrt{n}}\right)^n.$$
\end{lem}

\begin{rem}[Constants]
\label{num_rem}
We list here the numerology from Lemmas \ref{meas_lem}, \ref{construct_lem} and \ref{touch3_lem}.
We have
$$M_2=432n M_1; M_1=8 (36 n)^{\max\{1,  \frac{(n-1)\Lambda}{\lambda}-1\}}; \sigma=1-   \left(\frac{\lambda}{\lambda + (n-1)\Lambda}\right)^n  \left(\frac{1}{4\sqrt{n}}\right)^n;$$
$$M=256n^2 M_2=884736n^3  (36 n)^{\max\{1,  \frac{(n-1)\Lambda}{\lambda}-1\}}< 10^5  n^3(36n)^{\max\{1, \frac{(n-1)\Lambda}{\lambda}-1\}}.$$
\end{rem}

\section{Proofs of Theorems \ref{improve2} and \ref{improve3}}
\label{pf_sec}
In this section, we prove 
Theorems \ref{improve2} and \ref{improve3}. First, we recall a consequence of the Calder\'on-Zygmund cube decomposition (see \cite[Lemma 4.2]{CC}).
\begin{prop} \label{CZD}
Suppose that $D \subseteq E \subseteq Q_1\subset\R^n$ are measurable and $0< \delta < 1$ is such that:
\begin{itemize}
\item $|D| \leq \delta |Q_1|$; and
\item if $x\in \R^n$ and $r> 0$ such that $Q_{3r}(x) \subseteq Q_1$ and $|D\cap Q_{r}(x)| \geq \delta |Q_{r}(x)|$, then $Q_{3r}(x) \subseteq E$.
\end{itemize}
Then $|D| \leq \delta |E|$.
\end{prop}
Finally, we state our main $W^{2,\e}$ estimates from which Theorem \ref{improve2} follows.
\begin{thm}
\label{W2delta_thm}
Let $\lambda\leq\Lambda$ be positive constants. Let $\sigma$ and $M$ be as in Lemma \ref{touch3_lem}.
Assume that $\overline{B}_{2\sqrt{n}}\subset\Omega\subset\R^n.$
Suppose that $v\in C(\Omega)$ satisfies
 $\mathcal{M}^{-}_{\lambda,\Lambda} (D^2 v)\leq 0$ in $\Omega$
 with $|v|\leq 1/4.$
Then, for all $k=0, 1, \cdots,$ we have
\begin{equation}
\label{Adecay1}|A_{M^k}^{-}(v,\Omega)\cap Q_1|\leq \sigma^k.
\end{equation}
Therefore, for any $t>M$, we have
\begin{equation}
\label{Adecay2}
|A_{t}^{-}(v,\Omega)\cap Q_1|\leq \sigma^{-1}t^{-\frac{\log \frac{1}{\sigma}}{\log M}}.
\end{equation}
\end{thm}

\begin{proof}[Proof of Theorem \ref{W2delta_thm}]
In this proof, we use the following consequence of Lemma \ref{touch3_lem}: For $Q_{3r}(x_0)\subset Q_1$ and $t>0$, 
\begin{equation}
\label{touch4}
\text{if }
|A^{-}_{Mt}(v,\Omega)\cap Q_r(x_0)|>\sigma |Q_r(x_0)| ~\text{then } Q_{3r}(x_0)\subset A^{-}_t (v,\Omega).
\end{equation}
To obtain (\ref{touch4}), we apply Lemma \ref{touch3_lem} to $\tilde v$ in the domain $\tilde \Omega$
where
$$\ \tilde v (y)=\frac{1}{t r^2} v(x_0 + ry)~\text{for } y\in \tilde\Omega: =\varphi(\Omega) \text{ with }\varphi(x):= (x-x_0)/r.$$
From Lemmas \ref{touch1_lem} and \ref{touch3_lem},
we have
$|G_{M}^{-}(v,\Omega)\cap Q_1|\geq 1- \sigma.$
Hence, for all $k=0, 1, \cdots,$ we have, 
$$|A_{M^{k+1}}^{-}(v,\Omega)\cap Q_1|\leq\sigma.$$
To prove (\ref{Adecay1}), it suffices to show that for all $k=0,1,\cdots,$
$$|A_{M^{k+1}}^{-}(v,\Omega)\cap Q_1| \leq \sigma |A_{M^k}^{-}(v,\Omega)\cap Q_1|.$$
For each $k=0, 1, \cdots,$ let
$$A:= A_{M^{k+1}}^{-}(v,\Omega)\cap Q_1,B:=A_{M^k}^{-}(v,\Omega)\cap Q_1.$$
We claim that $|A|\leq\sigma |B|. $
To do this, we just note that if $Q= Q_r(x_0)$ is a cube in $Q_1$ such that $\tilde Q:= Q_{3r}(x_0)\subset Q_1$
and
$|A\cap Q|>\sigma |Q|$
then, by (\ref{touch4}), $\tilde Q\subset B$.
The claim follows from Proposition \ref{CZD} and hence (\ref{Adecay1}) is established.

Finally, let us prove (\ref{Adecay2}). For any $t>M$, there is a positive integer $k$ such that $M^k\leq t<M^{k+1}$. Hence $k+1>\frac{\log t}{\log M}.$ From this together with (\ref{Adecay1}), we get
\begin{eqnarray*}
|A_{t}^{-}(v,\Omega)\cap Q_1|\leq |A_{M^k}^{-}(v,\Omega)\cap Q_1|\leq \sigma^k=\sigma^{-1} \sigma^{k+1} <\sigma^{-1}\sigma^{\frac{\log t}{\log M}}= \sigma^{-1}t^{-\frac{\log \frac{1}{\sigma}}{\log M}}.
\end{eqnarray*}
\end{proof}

\begin{proof}[Proof of Theorem \ref{improve2}]
From Theorem \ref{W2delta_thm}, we conclude that 
\begin{equation} 
\label{Alocal}
\left|  A_t^{-}(u,B_1)\cap Q_{\frac{1}{2\sqrt{n}}} \right| \leq Ct^{-\e}
\end{equation}
for all $t>4M \sup_{B_1} |u|$,
with 
$$C=\sigma^{-1}|Q_{\frac{1}{2\sqrt{n}}} |, \text{ and }\e=\frac{\log \frac{1}{\sigma}}{\log M}.$$
The estimate (\ref{w2delta}) now follows from (\ref{Alocal})
and an easy covering argument. To obtain
the estimate for $\e$ as asserted in (\ref{e_est}), we use the fact that $\log (\frac{1}{a})>1-a$ for all $a\in (0, 1)$ together with the values for $\sigma$ and $M$ as recorded in Remark \ref{num_rem}.
We finally have
$$\e=\frac{\log \frac{1}{\sigma}}{\log M}>\frac{1-\sigma}{\log M}>\left(\frac{\lambda}{\lambda + (n-1)\Lambda}\right)^n\left(
\frac{1}{4\sqrt{n}}\right)^n\frac{1}{\log  \left[10^5 n^3 (36n)^{\max\{1, \frac{(n-1)\Lambda}{\lambda}-1\}}\right]}.$$
\end{proof}

\begin{proof}[Proof of Theorem \ref{improve3}]
The proof of Theorem \ref{improve3} is similar to that of Theorem \ref{improve2}. Instead of using Lemma \ref{meas_lem}, we use Lemma \ref{st_meas_lem}. We omit the details.
\end{proof}
{\bf Acknowledgements.} The author would like to thank the anonymous referee for the pertinent comments and the careful reading of the paper which help improve the exposition.

\end{document}